\documentclass[10pt,reqno]{amsart}
\newtheorem{theorem}{Theorem}
\newtheorem{lemma}[theorem]{Lemma}

\theoremstyle{definition}
\newtheorem{definition}{Definition}
\newtheorem{example}{Example}
\newtheorem{corollary}{Corollary}

\theoremstyle{remark}
\newtheorem{remark}{Remark}
\numberwithin{equation}{section}

\def\JS{\text{JS}}
\def\js{\text{js}}

\def\LS{\text{LS}}
\def\ls{\text{ls}}
\def\D{\mathcal{D}}
\def\Orb{\textrm{Orb}}

\def\rec{\text{rec}}
\def\Sing{\text{Sing}}
\def\N{\mathbb{N}}
\usepackage{latexsym, amsfonts, amsmath, amsthm, amssymb,amscd,color}
\usepackage[dvips]{graphicx}
\usepackage{eepic}
\usepackage{color}
\definecolor{gris25}{gray}{0.2}

\def\pr{\textrm{pr}}
\begin{document}

\title[Combinatorial interpretations  of the Jacobi-Stirling numbers]{Combinatorial interpretations  of the Jacobi-Stirling numbers}
\author{Yoann Gelineau}%    Address of record for the research reported here
\address{Universit\'{e} de Lyon \\ Universit\'{e} Lyon 1 \\ Institut Camille Jordan \\ UMR 5208 du CNRS \\
43, boulevard du 11 novembre 1918 \\ F-69622 Villeurbanne Cedex, France}
\email{gelineau@math.univ-lyon1.fr}
\urladdr{http://math.univ-lyon1.fr/~gelineau/}

\author{Jiang Zeng}
\address{Universit\'{e} de Lyon \\ Universit\'{e} Lyon 1 \\ Institut Camille Jordan \\ UMR 5208 du CNRS \\
43, boulevard du 11 novembre 1918 \\ F-69622 Villeurbanne Cedex, France}
\email{zeng@math.univ-lyon1.fr}
\urladdr{http://math.univ-lyon1.fr/~zeng/}



\begin{abstract}
The Jacobi-Stirling numbers of the first and second kinds were introduced in 2006 in
the spectral theory and are polynomial refinements of
the Legendre-Stirling numbers.
Andrews and Littlejohn have recently
given a combinatorial interpretation for the second kind of the latter numbers.
Noticing  that these numbers are very similar to the classical central factorial numbers,
we give combinatorial interpretations for the Jacobi-Stirling numbers of both kinds,
which provide a unified treatment of
the combinatorial theories for  the two previous sequences and also for
the Stirling numbers of both kinds.
\end{abstract}

\maketitle

%\tableofcontents

\section{Introduction}
\label{s:Intro}
It is well known that Jacobi polynomials $P_n^{(\alpha,\beta)}(t)$ satisfy the classical second-order Jacobi differential equation:
\begin{equation} \label{Jacobi} (1-t^2)y''(t)+(\beta-\alpha-(\alpha+\beta+2)t)y'(t)+n(n+\alpha+\beta+1)y(t)=0.
\end{equation}
Let $\ell_{\alpha,\beta}[y](t)$ be the Jacobi differential operator:
$$\ell_{\alpha,\beta}[y](t)=\frac{1}{(1-t)^\alpha(1+t)^\beta}\left( -(1-t)^{\alpha+1}(1+t)^{\beta+1} y'(t)\right)' .$$
Then, equation \eqref{Jacobi} is equivalent to say that $y=P_{n}^{(\alpha,\beta)}(t)$ is a solution of
$$\ell_{\alpha,\beta}[y](t)=n(n+\alpha+\beta+1)y(t).$$
In \cite[Theorem 4.2]{Everitt2}, for each $n \in \N$, Everitt et al. gave  the following expansion of the $n$-th composite power of $\ell_{\alpha,\beta}$:
$$(1-t)^\alpha(1+t)^\beta \ell_{\alpha,\beta}^n[y](t)=\sum\limits_{k=0}^{n} (-1)^k \left( P^{(\alpha,\beta)}S_{n}^{k}(1-t)^{\alpha+k}(1+t)^{\beta+k} y^{(k)}(t)\right)^{(k)},$$
where $P^{(\alpha,\beta)}S_{n}^{k}$ are called the Jacobi-Stirling numbers of the second kind.
They~\cite[(4.4)]{Everitt2} also gave an explicit summation formula for $P^{(\alpha,\beta)}S_{n}^{k}$ numbers, showing that these numbers depend only on one parameter $z=\alpha+\beta+1$. So we can define  the Jacobi-Stirling numbers as the connection coefficients in the following equation:
\begin{align}\label{eq:defJS}
x^n=\sum_{k=0}^n {\JS}_n^k(z)\prod_{i=0}^{k-1}(x-i(z+i)),
\end{align}
where $\JS_{n}^{k}(z)=P^{(\alpha,\beta)}S_{n}^{k}$, while the Jacobi-Stirling numbers of the first kind
 can be defined by inversing the above equation:
\begin{align}\label{eq:defjs}
\prod_{i=0}^{n-1}(x-i(z+i))=\sum_{k=0}^n {\js}_n^k(z)x^k,
\end{align}
where $\js_n^k(z)=P^{(\alpha,\beta)}s_n^k$ in the notations of \cite{Everitt2}.

It follows from \eqref{eq:defJS} and \eqref{eq:defjs} that the Jacobi-Stirling numbers  $\JS_n^k( z)$ and  $\js_n^k( z)$
 satisfy, respectively,
  the following recurrence relations:
\begin{equation} \label{eqlnk1}
\left\{ \begin{array}{l} \JS_0^0(z)=1, \qquad \JS_n^k(z)=0, \quad  \text{if} \ k \not\in\{1,\ldots,n\}, \\
\JS_n^k(z)= \JS_{n-1}^{k-1}(z)+k(k+z)\,\JS_{n-1}^{k}(z),  \quad   n,k \geq 1. \end{array} \right.
 \end{equation}
 and
\begin{equation} \label{eqlnk2}
\left\{ \begin{array}{l} \js_0^0(z)=1, \qquad \js_n^k(z)=0, \quad  \text{if} \ k \not\in\{1,\ldots,n\}, \\
\js_n^k(z)= \js_{n-1}^{k-1}(z)-(n-1)(n-1+z)\,\js_{n-1}^{k}(z),  \quad   n,k \geq 1. \end{array} \right.
 \end{equation}
 The first values of $\JS_{n}^{k}(z)$ and
$\js_{n}^{k}(z)$ are given, respectively,  in  Tables \ref{Lnkz} and \ref{jsnkz}.
\begin{table}
\caption{The first values of $\JS_n^k(z)$}
\label{Lnkz}
\begin{tiny} \[ \begin{tabular}{c|cccccc}
$k\backslash n$ & $1$ & $2$ & $3$ & $4$ & $5$  & $6$
\\
\hline
$1$ & $1$&$z+1$&$(z+1)^2$& $(z+1)^3$  &$(z+1)^4$&$(z+1)^5$ \\
$2$ && $1$& $5+3z$&$21+24z+7z^2$  &$85+141z+79z^2+15z^3$&$341+738z+604z^2+222z^3+31z^4$   \\
$3$ & &&$1$ &$14+6z$&$147+120z+25z^2$ &$1408+1662z+664z^2+90z^3$   \\
$4$ & &&&$1$ &  $30+10z$ &$627+400z+65z^2$                \\
$5$ & &&&&$1$  &$55+15z$        \\
$6$ & & &&&&$1$ \\
 \end{tabular} \] \end{tiny} \end{table}

As remarked in \cite{Everitt1,Everitt2,AL09}, the previous definitions are reminiscent to the well-known
 \emph{Stirling numbers}
 of the second (resp. the first) kind $S(n,k)$ (resp. $s(n,k)$),
which are defined (see \cite{Comtet74}) by
$$
x^n=\sum_{k=0}^nS(n,k) \prod_{i=0}^{k-1}(x-i),\qquad
\prod_{i=0}^{n-1}(x-i)=\sum_{k=0}^ns(n,k)x^k.
$$
and satisfy the following recurrences:
\begin{align}
S(n,k) &= S(n-1,k-1)+k S(n-1,k), \quad &n,k \geq 1,\label{eqSt2}\\
s(n,k) &= s(n-1,k-1)-(n-1) s(n-1,k), \quad &n,k \geq 1.\label{eqSt1}
\end{align}

The starting point of this paper is the observation  that the
\emph{central factorial numbers} of the second (resp. the first)
kind $T(n,k)$ (resp. $t(n,k)$) seem to be more appropriate for comparaison.
Indeed, these numbers
  are defined in Riordan's book \cite[p. 213-217]{Riordan} by
\begin{align}
x^n=\sum_{k=0}^nT(n,k)\,x\prod_{i=1}^{k-1}\left(x+\frac{k}{2}-i\right),
\end{align}
and
\begin{align} \label{eqtnk}
x\prod_{i=1}^{n-1}\left(x+\frac{n}{2}-i\right)=\sum_{k=0}^nt(n,k)x^k.
\end{align}
Therefore, if we denote the central factorial numbers of even indices by $U(n,k)=T(2n,2k)$ and $u(n,k)=t(2n,2k)$, then~:

\begin{align}
U(n,k)&=U(n-1,k-1)+k^2U(n-1,k),\label{eqUnk}  \\
 u(n,k)&=u(n-1,k-1)-(n-1)^2u(n-1,k). \label{equnk}
\end{align}

\begin{table}
\caption{The first values of $js_n^k(z)$}
\label{jsnkz}
\begin{tiny} \[ \begin{tabular}{c|ccccc}
$k\backslash n$ & $1$ & $2$ & $3$ & $4$ & $5$
\\
\hline
$1$ & $1$       & $-z-1$   & $2z^2+6z+4$    & $-6z^3-36z^2-66z-36$ & $24z^4+240z^3+840z^2+1200z+576$
\\ $2$ &  & $1$ & $-3z-5$  & $11z^2+48z+49$ & $-50z^3-404z^2-1030z-820$
\\ $3$ & & &$1$ & $-6z-14$ & $35z^2+200z+273$
 \\$4$ & &&&$1$ & $-10z-30$
 \\$5$ & &&&&$1$
\\
 \end{tabular} \] \end{tiny} \end{table}

From  \eqref{eqlnk1}-\eqref{equnk}, we easily derive  the following result.

\begin{theorem} \label{thm1} Let $n,k$ be positive integers with $n \geq k$.
The Jacobi-Stirling numbers $\JS_n^k(z)$ and $(-1)^{n-k}\js_{n}^{k}(z)$
are  polynomials in $z$ of degree $n-k$ with positive integer coefficients.
Moreover, if
\begin{align}\label{eq:defa}
\JS_{n}^{k}(z)&=a_{n,k}^{(0)}+a_{n,k}^{(1)}z+\cdots +a_{n,k}^{(n-k)}z^{n-k},\\
(-1)^{n-k}\js_{n}^{k}(z)&=b_{n,k}^{(0)}+b_{n,k}^{(1)}z+\cdots +b_{n,k}^{(n-k)}z^{n-k},\label{eq:defb}
\end{align}
then
$$%\begin{array}{llll}
 a_{n,k}^{(n-k)}=S(n,k),\quad a_{n,k}^{(0)}=U(n,k), \quad
b_{n,k}^{(n-k)}=|s(n,k)|,  \quad b_{n,k}^{(0)}=|u(n,k)|.
%\LS(n,k)&=\displaystyle \sum_{i=0}^{n-k}a_{n,k}^{(i)},\quad
%&|\ls(n,k)|&=\displaystyle \sum_{i=0}^{n-k}b_{n,k}^{(i)}.
%\end{array}
$$
\end{theorem}

Note that when $z=1$, the Jacobi-Stirling numbers reduce to
the \emph{Legendre-Stirling numbers} of
the first and the second kinds \cite{Everitt1}:
\begin{align}\label{defLS}
\LS(n,k)=\JS_n^k(1),\quad \ls(n,k)=\js_n^k(1).
\end{align}

The integral nature of the involved coefficients in the above polynomials ask for combinatorial interpretations.
Indeed, it is folklore (see \cite{Comtet74}) that the Stirling number $S(n,k)$ (resp. $|s(n,k)|$) counts the number of
 partitions (resp. permutations) of $[n]:=\{1,\ldots,n\}$ into  $k$ blocks (resp. cycles). In 1974, in his study of Genocchi numbers,
Dumont \cite{Dumont} discovered
 the first combinatorial interpretation for the central factorial number $U(n,k)$ in terms of
ordered pairs of supdiagonal quasi-permutations of $[n]$ (cf. \S~2).
Recently, Andrews and Littlejohn~\cite{AL09}
interpreted $\JS_n^k(1)$ in terms of  set partitions (cf. \S~2).

Several questions arise naturally in the light of the above  known results:
\begin{itemize}
\item First of all, what is the combinatorial refinement of Andrews and Littlejohn's model
which gives the combinatorial counterpart for the coefficient $a_{n,k}^{(i)}$?
\item Secondly, is there any  connection between  the model of Dumont
and  that of Andrews and Littlejohn?
\item Thirdly, is there any combinatorial interpretation for the coefficient
$b_{n,k}^{(i)}$ in the Jacobi-Stirling numbers of the  first kind, generalizing that for the Stirling number  $|s(n,k)|$?
\end{itemize}

The aim of this paper is to settle all of these questions. Additional results of the same type are also provided.

In Section 2, after introducing some necessary definitions, we give two combinatorial
interpretations  for the coefficient $a_{n,k}^{(i)}$ in
$\JS_{n}^{k}(z)$ ($0 \leq i \leq n-k$), and explicitly construct a bijection between the two models. In Section 3, we give a combinatorial interpretation for the coefficient $b_{n,k}^{(i)}$ in
$\js_{n}^{k}(z)$ ($0 \leq i \leq n-k$). In Section 4, we give the combinatorial interpretation for two sequences which are multiples of the central factorial numbers of odd indices and we also establish a simple derivation of the explicit formula of Jacobi-Stirling numbers.

\section{Jacobi-Stirling numbers of the second kind $\JS_{n}^{k}(z)$}

\subsection{First interpretation}
\label{s:interpretation} For any positive integer $n$, we define
$$
[\pm n]_0:= \{0, 1,-1,2,-2,3,-3,\ldots,n,-n\}.
$$
The following definition is equivalent to that given by Andrews and Littlejohn~\cite{AL09} in order
to interpret Legendre-Stirling numbers, where 0 is added to avoid empty block and also to
be consistent with the model  for the Jacobi-Stirling numbers of the first kind.
\begin{definition}
A \emph{signed $k$-partition} of $[\pm n]_0$ is a set partition
of $[\pm n]_0$ with $k+1$ non-empty blocks $B_0, B_1,\ldots B_k$ with the following rules:
\begin{enumerate}
\item $0\in B_0$ and $ \forall i \in [n], \ \{ i,-i \} \not\subset B_0$,
\item $\forall j\in [k]$ and $\forall i\in [n]$,  we have
$\{i,-i\}\subset B_j\Longleftrightarrow i=\min B_j\cap [n]$.
\end{enumerate}
\end{definition}

For example, the partition $\pi=\{\{  2, -5\}_0 , \{ \pm 1 , -2 \} , \{ \pm 3 \}, \{ \pm 4, 5 \}\}$
 is a signed 3-partition of $[\pm 5]_0$,
 with  $\{ 2,-5 \}_0:=\{0,2,-5\}$ being  the zero-block.

\begin{theorem} \label{thm2} For any  positive integers  $n$ and $k$,
the integer $a_{n,k}^{(i)}$ $(0\leq i\leq n-k)$ is the number of signed $k$-partitions of $[\pm n]_0$
such that the zero-block contains $i$ signed entries.
 \end{theorem}
\begin{proof}
Let $\mathcal{A}_{n,k}^{(i)}$ be the set of signed $k$-partitions of $[\pm n]_0$
 such that the zero-block contains $i$ signed entries and $\tilde a_{n,k}^{(i)}=|\mathcal{A}_{n,k}^{(i)}|$.
 By convention  $\tilde a_{0,0}^{(0)}=1$.
 Clearly
 $\tilde a_{1,1}^{(0)}=1$ and
for  $\tilde a_{n,k}^{(i)}\neq 0$ we must have $n\geq k\geq 1$ and  $0\leq i\leq n-k$.
We divide  $\mathcal{A}_{n,k}^{(i)}$ into four parts:
\begin{itemize}
\item[(i)] the signed $k$-partitions of $[\pm n]_0$ with $\{ -n , n \}$ as a block.
Clearly, the number of such partitions is $\tilde a_{n-1,k-1}^{(i)}$.
\item[(ii)]  the signed $k$-partitions of $[\pm n]_0$ with   $n$
in the zero-block.  We can construct such partitions by first
constructing  a signed $k$-partition of $[\pm(n-1)]_0$  with $i$ signed entries in the zero block and then insert $n$ into the zero block and
$-n$ into one of the $k$ other blocks; so there are $k \tilde a_{n-1,k}^{(i)}$ such partitions.
\item[(iii)]  the signed $k$-partitions of $[\pm n]_0$ with $-n$ in  the zero-block.
We can construct such partitions by first  constructing a signed $k$-partition of $[\pm (n-1)]_0$ with $i-1$ signed entries in the zero-block, and then placing $n$ into one of the $k$ non-empty blocks, so there are $k \tilde a_{n-1,k}^{(i-1)}$ possibilities.
\item[(iv)]  the signed $k$-partitions of $[\pm n]_0$ where neither $n$ nor $-n$ appears in the zero-block and  $\{-n,n\}$ is not  a block.
We can construct such partitions by first  choosing a signed $k$-partition of   $[\pm (n-1)]_0$ with $i$
signed entries in the zero block, and then placing $n$ and $-n$ into two different
 non-zero blocks, so there are $k(k-1) \tilde a_{n-1,k}^{(i)}$ possibilities.
\end{itemize}
Summing up  we get the following equation:
  \begin{equation} \label{eq:aux}
  \tilde a_{n,k}^{(i)} = \tilde a_{n-1,k-1}^{(i)} + k \tilde a_{n-1,k}^{(i-1)} + k^2 \tilde a_{n-1,k}^{(i)}.
  \end{equation}
 By  \eqref{eqlnk1}, it is easy to see
 that  $a_{n,k}^{(i)}$ satisfies the same recurrence and initial conditions as $ \tilde a_{n,k}^{(i)}$, so
 they agree.
   \end{proof}

Since $\LS(n,k)=\sum_{i=0}^{n-k}a_{n,k}^{(i)}$, Theorem~\ref{thm2} implies immediately the following result of
Andrews and Littlejohn  \cite{AL09}.

\begin{corollary}
The integer $\LS(n,k)$ is the number of signed $k$-partitions of $[\pm n]_0$.
\end{corollary}

By Theorems~\ref{thm1}
 and \ref{thm2},  we derive  that the integer $S(n,k)$ is the number of
signed $k$-partitions of $[\pm n]_0$ such that the zero-block contains $n-k$ signed entries.
By definition, in this case, there is no positive entry in the zero-block.
By deleting the signed entries in the remaining $k$ blocks, we recover then the following known
interpretation for the Stirling number of the second kind.
\begin{corollary} \label{coroSnk} The integer $S(n,k)$ is the number of partitions of $[n]$ in $k$ blocks. \end{corollary}

For a partition $\pi=\{B_1,B_2,\ldots,B_k\}$ of $[n]$ in $k$ blocks, denote by $\min \pi$ the set of minima of blocks $$\min \pi =\{ \min(B_1), \ldots, \min(B_k) \}.$$
The following partition version of Dumont's interpretation for the central factorial number of even indices can be found  in \cite[Chap. 3]{FH}.
\begin{corollary} \label{coroUnk} The integer $U(n,k)$ is the number of ordered pairs $(\pi_1,\pi_2)$ of partitions of $[n]$ in $k$ blocks such that $\min( \pi_1)=\min(\pi_2)$. \end{corollary}

\begin{proof} As  $U(n,k)=a_{n,k}^{(0)}$,
by Theorem \ref{thm2}, the integer $U(n,k)$ counts the number of signed $k$-partitions of $[\pm n]_0$ such that the zero-block doesn't contain any signed entry.
For any such a signed $k$-partition $\pi$,  we apply the following
algorithm:  (i)  move each positive entry $j$ of the zero-block into the block containing
 $-j$ to obtain a signed $k$-partition $\pi'=\{\{0\}, B_1, \ldots, B_k\}$, (ii) $\pi_1$ is obtained by deleting the negative entries in each block $B_i$ of $\pi'$, and $\pi_2$ is obtained by deleting the positive entries and taking the opposite values of signed entries in each block of $\pi'$.  For example,
if  $\pi=\{\{3\}_0,\{\pm1,-3,4\},\{\pm 2,-4\}\}$ is
 the signed $2$-partition of $[\pm4]_0$, the corresponding ordered pair of partitions is $(\pi_1,\pi_2)$ with $\pi_1=\{\{1,3,4\},\{2\} \}$ and $\pi_2=\{\{1,3\},\{2,4\}\}$.\end{proof}

%%%%%%%%%%%%%%%%%%%%%%%%%%%%%%%%%%%%%%%%%%%%%%

The following result shows that the coefficients in the expansion of the Jacobi-Stirling numbers $\JS_n^{k}(z)$ in the
  basis $\{(z+1)^i\}_{i=0,\ldots,n-k}$ are also interesting.

\begin{table}
\caption{The first values of $\JS_n^k(z)$ in the basis $\{(z+1)^{i}\}_{i=0,\ldots, n-k}$}
\label{Lnkz2}
\begin{tiny} \[ \begin{tabular}{c|ccccc}
$k\backslash n$ & $1$ & $2$ & $3$ & $4$ & $5$  %& $6$
\\
\hline
$1$ & $1$      &$(z+1)$     &$(z+1)^2$               & $(z+1)^3$                         &$(z+1)^4$                            %&$(z+1)^5$
\\ $2$ && $1$     & $2+3(z+1)$ &$4+10(z+1)+7(z+1)^2$    & $8+28(z+1)+34(z+1)^2+15(z+1)^3$   %&$16+72(z+1)+124(z+1)^2+98(z+1)^3+31(z+1)^4$
\\ $3$ & &&$1$    &$8+6(z+1)$  &$52+70(z+1)+25(z+1)^2$  % &$320+604(z+1)+394(z+1)^2+90(z+1)^3$
 \\
$4$ & &&&$1$   &$20+10(z+1)$   %&$292+270(z+1)+65(z+1)^2$
\\
$5$ & &&&&$1$ % &$40+15(z+1)$
\\
% $6$ & & &&&&$1$  \\
 \end{tabular} \] \end{tiny} \end{table}

\begin{theorem} Let
\begin{align}\label{z+1}
\JS_{n}^{k}(z)=d_{n,k}^{(0)}+d_{n,k}^{(1)}(z+1)+\cdots +d_{n,k}^{(n-k)}(z+1)^{n-k}.
\end{align} Then the coefficient $d_{n,k}^{(i)}$ is a positive integer, which counts the number of
signed $k$-partitions of $[\pm n]_0$ such that the zero-block contains only
zero and  $i$ negative values. \end{theorem}

\begin{proof} We derive from \eqref{eqlnk1}
that the coefficients $d_{n,k}^{(i)}$ verify the following recurrence relation:
\begin{equation} \label{eqdnk}
d_{n,k}^{(i)} = d_{n-1,k-1}^{(i)} + k d_{n-1,k}^{(i-1)} + k(k-1) d_{n-1,k}^{(i)}.
\end{equation}
As for the $a_{n,k}^{(i)}$, we can prove the result by a similar argument as in proof of
Theorem~\ref{thm2}.
\end{proof}

\begin{corollary} The integer  $J_{n}^{k}(-1) = d_{n,k}^{(0)}$ is the number of
  signed $k$-partitions of $[\pm n]_0$ with $\{0\}$ as zero-block.
   \end{corollary}

\begin{remark}
A priori, it was not obvious that $J_{n}^{k}(-1)$ $=\sum\limits_{i=i}^{n-k} (-1)^i a_{n,k}^{(i)}$ was positive.
\end{remark}

From Theorem~\ref{thm1} and \eqref{z+1}, we derive the following relations~:
\begin{equation} \label{eqankdnk}
 a_{n,k}^{(i)} = \sum\limits_{j=i}^{n-k} \binom{j}{i} d_{n,k}^{(j)},\quad U(n,k)=\sum\limits_{j=i}^{n-k} d_{n,k}^{(j)},\quad
 \LS(n,k)=\sum\limits_{j=i}^{n-k} 2^j d_{n,k}^{(j)}.
  \end{equation}
 We can give  combinatorial interpretations for these formulas. For example, for the first one, we can split the set $\mathcal{A}_{n,k}^{(i)}$ by counting the total number $j$ of elements in the zero-block ($1 \leq j \leq n-k$). Then to construct such an element, we first take a signed $k$-partition of $[\pm n]_0$
with no positive values in the zero-block, so there are $d_{n,k}^{(j)}$ possibilities, and then we choose the $j-i$ numbers that are positive among the $j$ possibilities in the zero-block. Similar proofs can be easily described for the two other formulas.
\medskip

%%%%%%%%%%%%%%%%%%%%%%%%%%%%%%%%%%%%%%%%
\subsection{Second interpretation}
%%%%%%%%%%%%%%%%%%%%%%%%%%%%%%%
We propose now a second model for the coefficient $a_{n,k}^{(i)}$, inspired by Foata and Sch\"{u}tzenberger \cite{FS72} and Dumont  \cite{Dumont}. Let $\mathcal{S}_n$ be the set of permutations of $[n]$. In the rest of this paper, we identify any permutation $\sigma$ in $\mathcal{S}_n$ with its diagram
${\mathcal D}(\sigma)=\{ (i,\sigma(i)): i \in [n] \}$.

For any finite set $X$, we denote by
$|X|$ its  cardinality.  If $\alpha=(i,j)\in [n]\times [n]$, we define $\pr_x(\alpha)=i$ and
 $\pr_y(\alpha)=j$ to be its $x$ and $y$ projections.
For any subset  $Q$ of $[n]\times [n]$,
we define the $x$ and $y$ projections by
\begin{align*}
\pr_x(Q)=\{\pr_x(\alpha): \alpha \in Q\},\qquad \pr_y(Q)=\{\pr_y(\alpha): \alpha\in Q\};
\end{align*}
and the supdiagonal and subdiagonal parts by
$$
Q^+=\{(i,j)\in Q: i\leq j\}, \qquad Q^{-}=\{(i,j)\in Q: i\geq j\}.
$$

\begin{definition} A \emph{simply hooked $k$-quasi-permutation} of $[n]$
is a subset $Q$ of $[n]\times [n]$ such that
\begin{itemize}
\item[i)] $Q \subset \D(\sigma)$ for some permutation  $\sigma$ of $[n]$,
\item[ii)] $|Q|=n-k$ and $\pr_x(Q^-)\cap \pr_y(Q^+)=\emptyset$.
\end{itemize}
\end{definition}
%such that $|Q \cap H_i| \leq 1$ for $i=1..n$,
% where $H_i$ is the $i$-th diagonal hook of $Q$ (see Figure \ref{hook}).
\begin{figure}
\caption{\label{hook}The diagonal hook $H_4$  and a   simply hooked quasi-permutation of $[6]$: $Q=\{(2,5),(4,2),(6,3)\}$}
{\setlength{\unitlength}{0.8mm}
\begin{picture}(32,32)(0,0)
%%%%%%%%%%%%%%%%%%%%%%%%%%%%%%%%%%%%%%%% coloriage
\put(25,10){\color{gris25}{\rule{4mm}{4mm}}}
\put(5,20){\color{gris25}{\rule{4mm}{4mm}}}
\put(15,5){\color{gris25}{\rule{4mm}{4mm}}}
%%%%%%%%%%%%%%%%%%%%%%%%%%%%%%%%%%%%%%%% QUADRILLAGE
\put(0,0){\line(1,0){30}}
\put(0,5){\line(1,0){30}}
\put(0,10){\line(1,0){30}}
\put(0,15){\line(1,0){30}}
\put(0,20){\line(1,0){30}}
\put(0,25){\line(1,0){30}}
\put(0,30){\line(1,0){30}}
\put(0,0){\line(0,1){30}}
\put(5,0){\line(0,1){30}}
\put(10,0){\line(0,1){30}}
\put(15,0){\line(0,1){30}}
\put(20,0){\line(0,1){30}}
\put(25,0){\line(0,1){30}}
\put(30,0){\line(0,1){30}}
%%%%%%%%%%%%%%%%%%%%%%%%%%%%%%%%%%%%%%%%% TRACES EN GRAS
\linethickness{0.7mm}
\put(0,20){\line(1,0){20}}
\put(0,15){\line(1,0){15}}
\put(0,15){\line(0,1){5}}
\put(20,0){\line(0,1){20}}
\put(15,0){\line(0,1){15}}
\put(15,0){\line(1,0){5}}
%%%%%%%%%%%%%%%%%%%%%%%%%%%%%%%%%%%%%%%% ELIMINATION SUD/EST
%\put(0,5){\line(1,-1){5}}
%\put(5,10){\line(1,-1){5}}
%\put(10,15){\line(1,-1){5}}
%\put(15,20){\line(1,-1){5}}
%\put(0,0){\line(1,1){20}}
\end{picture}}
\end{figure}
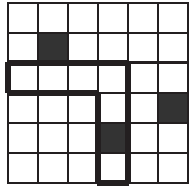

A simply hooked $k$-quasi-permutation $Q$ of $[n]$
 can be depicted by darkening  the $n-k$ corresponding boxes of $Q$ in the $n \times n$ square tableau.
Conversely, if we define the diagonal hook $ H_i:=\{(i,j): i\leq j\}\cup \{(j,i): i\leq j\}$ ($1\leq i\leq n$),
% $$
% H_i:=\{(i,j)|i\leq j\}\cup \{(j,i)|i\leq j\},
% $$
then a black subset of  the $n \times n$ square tableau represents a  simply hooked quasi-permutation if
there is no black box on  the main diagonal  and at most  one black
 box in each row, in each column and in each  diagonal hook.
 An example is given in  Figure \ref{hook}.

\begin{theorem} \label{mainthm} The integer  $a_{n,k}^{(i)}$ $(1\leq i\leq n-k)$ is the number of
 ordered pairs $(Q_1,Q_2)$ of simply hooked $k$-quasi-permutations of $[n]$ satisfying the following conditions:
 \begin{align}
 Q_1^-=Q_2^-,\quad
 |Q_1^-|=|Q_2^-|=i \quad\textrm{and}\quad\pr_y(Q_1)=\pr_y(Q_2).\label{eq:quasi}
 \end{align}
\end{theorem}

\begin{proof}
Let  $\mathcal{C}_{n,k}^{(i)}$ be the set of ordered pairs $(Q_1,Q_2)$ of simply hooked $k$-quasi-permutations of $[n]$ verifying
\eqref{eq:quasi}, and let $c_{n,k}^{(i)}=|\mathcal{C}_{n,k}^{(i)}|$.

For example, the ordered pair $(Q_1,Q_2)$ with
\begin{equation}\label{eq:qq}
\begin{split}
Q_1&=\{ (1,3),(2,5),(3,7),(4,1),(5,6),(8,2),(10,9) \},\\
Q_2&=\{ (1,5),(2,3),(3,6),(4,1),(5,7),(8,2),(10,9) \},
\end{split}
\end{equation}
 is an element of $\mathcal{C}_{10,3}^{(3)}$. A graphical representation is given in  Figure \ref{couple}.

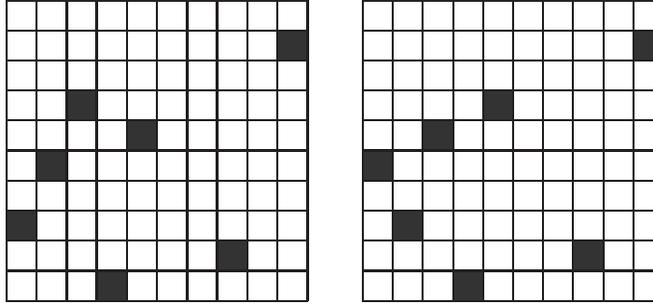
\begin{figure}
\caption{\label{couple}An ordered pair of simply hooked quasi-permutations  in $\mathcal{C}_{10,3}^{(3)}$}
{\setlength{\unitlength}{0.8mm}
\begin{picture}(50,50)(0,0)
%%%%%%%%%%%%%%%%%%%%%%%%%%%%%%%%%%%%%%%% coloriage
\put(0,10){\color{gris25}{\rule{4mm}{4mm}}}
\put(5,20){\color{gris25}{\rule{4mm}{4mm}}}
\put(10,30){\color{gris25}{\rule{4mm}{4mm}}}
\put(15,0){\color{gris25}{\rule{4mm}{4mm}}}
\put(20,25){\color{gris25}{\rule{4mm}{4mm}}}
\put(35,5){\color{gris25}{\rule{4mm}{4mm}}}
\put(45,40){\color{gris25}{\rule{4mm}{4mm}}}
%%%%%%%%%%%%%%%%%%%%%%%%%%%%%%%%%%%%%%%% QUADRILLAGE
\put(0,0){\line(1,0){50}}
\put(0,5){\line(1,0){50}}
\put(0,10){\line(1,0){50}}
\put(0,15){\line(1,0){50}}
\put(0,20){\line(1,0){50}}
\put(0,25){\line(1,0){50}}
\put(0,30){\line(1,0){50}}
\put(0,35){\line(1,0){50}}
\put(0,40){\line(1,0){50}}
\put(0,45){\line(1,0){50}}
\put(0,50){\line(1,0){50}}
\put(0,0){\line(0,1){50}}
\put(5,0){\line(0,1){50}}
\put(10,0){\line(0,1){50}}
\put(15,0){\line(0,1){50}}
\put(20,0){\line(0,1){50}}
\put(25,0){\line(0,1){50}}
\put(30,0){\line(0,1){50}}
\put(35,0){\line(0,1){50}}
\put(40,0){\line(0,1){50}}
\put(45,0){\line(0,1){50}}
\put(50,0){\line(0,1){50}}
%%%%%%%%%%%%%%%%%%%%%%%%%%%%%%%%%%%%%%%% ELIMINATION SUD/EST
%\put(0,5){\line(1,-1){5}}
%\put(5,10){\line(1,-1){5}}
%\put(10,15){\line(1,-1){5}}
%\put(15,20){\line(1,-1){5}}
%\put(20,25){\line(1,-1){5}}
%\put(25,30){\line(1,-1){5}}
%\put(0,0){\line(1,1){30}}
\end{picture} \hspace{0.5cm} \begin{picture}(50,50)(0,0)
%%%%%%%%%%%%%%%%%%%%%%%%%%%%%%%%%%%%%%%% coloriage
\put(0,20){\color{gris25}{\rule{4mm}{4mm}}}
\put(5,10){\color{gris25}{\rule{4mm}{4mm}}}
\put(10,25){\color{gris25}{\rule{4mm}{4mm}}}
\put(15,0){\color{gris25}{\rule{4mm}{4mm}}}
\put(20,30){\color{gris25}{\rule{4mm}{4mm}}}
\put(35,5){\color{gris25}{\rule{4mm}{4mm}}}
\put(45,40){\color{gris25}{\rule{4mm}{4mm}}}
%%%%%%%%%%%%%%%%%%%%%%%%%%%%%%%%%%%%%%%% QUADRILLAGE
\put(0,0){\line(1,0){50}}
\put(0,5){\line(1,0){50}}
\put(0,10){\line(1,0){50}}
\put(0,15){\line(1,0){50}}
\put(0,20){\line(1,0){50}}
\put(0,25){\line(1,0){50}}
\put(0,30){\line(1,0){50}}
\put(0,35){\line(1,0){50}}
\put(0,40){\line(1,0){50}}
\put(0,45){\line(1,0){50}}
\put(0,50){\line(1,0){50}}
\put(0,0){\line(0,1){50}}
\put(5,0){\line(0,1){50}}
\put(10,0){\line(0,1){50}}
\put(15,0){\line(0,1){50}}
\put(20,0){\line(0,1){50}}
\put(25,0){\line(0,1){50}}
\put(30,0){\line(0,1){50}}
\put(35,0){\line(0,1){50}}
\put(40,0){\line(0,1){50}}
\put(45,0){\line(0,1){50}}
\put(50,0){\line(0,1){50}}
%%%%%%%%%%%%%%%%%%%%%%%%%%%%%%%%%%%%%%%% ELIMINATION SUD/EST
%\put(0,5){\line(1,-1){5}}
%\put(5,10){\line(1,-1){5}}
%\put(10,15){\line(1,-1){5}}
%\put(15,20){\line(1,-1){5}}
%\put(20,25){\line(1,-1){5}}
%\put(25,30){\line(1,-1){5}}
%\put(0,0){\line(1,1){30}}
\end{picture}} \end{figure}

We divide the set $\mathcal{C}_{n,k}^{(i)}$ into three parts:
\begin{itemize}
\item the ordered pairs $(Q_1,Q_2)$ such that  the $n$-th rows and $n$-th columns of
$Q_1$ and $Q_2$  are empty. Clearly, there are $c_{n-1,k-1}^{(i)}$  such elements.
\item
the ordered pairs $(Q_1,Q_2)$ such that  the  $n$-th columns of
$Q_1$ and $Q_2$  are not empty.
We can first construct an ordered pair $(Q_1',Q_2')$ of $\mathcal{C}_{n-1,k}^{(i-1)}$  and then choose a
box in  the same position of
 the $n$-th column  of  both simply hooked quasi-permutations, there are
 $n-1-(n-k-1)=k$ positions available.   So there are
$k c_{n-1,k}^{(i-1)}$ such elements.
\item the ordered pairs $(Q_1,Q_2)$ such that  the $n$-th rows of $Q_1$ and $Q_2$
are not empty.
We can first construct an ordered pair $(Q_1',Q_2')$ of $\mathcal{C}_{n-1,k}^{(i)}$  and then add
a black box in the top of  both simply hooked quasi-permutations, the box can be placed
on any of the $n-1-(n-k-1)=k$ positions whose columns are  empty. So there are
$k^2 c_{n-1,k}^{(i)}$ such elements.
\end{itemize}
In conclusion, we obtain the recurrence
 \begin{equation} \label{eqcnk}
 c_{n,k}^{(i)} = c_{n-1,k-1}^{(i)} + k c_{n-1,k}^{(i-1)} + k^2 c_{n-1,k}^{(i)}.
  \end{equation}
By  \eqref{eqlnk1}, we see that  $a_{n,k}^{(i)}$ satisfies the same
recurrence relation and the initial conditions as $ c_{n,k}^{(i)}$, so they agree.
  \end{proof}

\begin{remark} In the first model, we don't have a direct interpretation for the integer $k^2$ in \eqref{eq:aux} because it results from after the simplification $k+k(k-1)=k^2$. While in the second one, we can
see what the coefficient $k^2$ counts in \eqref{eqcnk}.
\end{remark}

\begin{definition} A  \emph{supdiagonal}  (resp. \emph{subdiagonal})
\emph{quasi-permutation} of $[n]$ is a simply hooked  quasi-permutation $Q$ of $[n]$
with  $Q^-=\emptyset$ (resp. $Q^+=\emptyset$).
\end{definition}

From Theorems~\ref{thm1} and \ref{mainthm},
 we recover Dumont's combinatorial interpretation for  the central factorial numbers of the second kind \cite{Dumont},
 and Riordan's interpretation for the Stirling numbers of the second kind (see \cite[Prop. 2.7]{FS72}).

\begin{corollary} \label{coroTnk} The integer $U(n,k)$ is the number of ordered pairs $(Q_1,Q_2)$ of supdiagonal $k$-quasi-permutations of $[n]$ such that $pr_y(Q_1)=pr_y(Q_2)$. \end{corollary}

\begin{corollary}  The integer
$S(n,k)$ is the number of subdiagonal (resp. supdiagonal) $k$-quasi-permutations of $[n]$.
\end{corollary}

\begin{remark}\label{remark3}
To recover the classical  interpretation of $S(n,k)$ in  Corollary~\ref{coroSnk},
we can apply a simple  bijection, say $\varphi$,  in \cite[Prop. 3]{FS72}.
Starting from a $k$-partition  $\pi=\{B_1,\ldots,B_k\}$  of $[n]$,
for each non-singleton block  $B_i=\{ p_1, p_2, \ldots , p_{n_i} \}$ with $n_i \geq 2$ elements $p_1 < p_2 < \ldots < p_{n_i}$, we associate the subdiagonal quasi-permutation
$Q_i=\{ (p_{n_i},p_{n_{i-1}}),(p_{n_{i-1}},p_{n_{i-2}}),\ldots,(p_2,p_1) \}$
with $n_i-1$ elements of $[n] \times [n]$. Clearly, the union  of all such  $Q_i's$
 is a subdiagonal quasi-permutation of cardinality $n-k$. An example of the map $\varphi$ is given in Figure~\ref{permutation}.
 \end{remark}

\begin{figure}
\caption{ \label{permutation}
The subdiagonal quasi-permutation corresponding to a partition via the map $\varphi$}
\begin{picture}(30,30)(60,-25)
\put(15,15){$\pi=\{ \{1,4,6\}, \{2,5\}, \{3\} \} \ \longrightarrow$}
\end{picture}
{\setlength{\unitlength}{0.8mm}
\begin{picture}(30,30)(-30,1)
%%%%%%%%%%%%%%%%%%%%%%%%%%%%%%%%%%%%%%%% coloriage
\put(15,0){\color{gris25}{\rule{4mm}{4mm}}}
\put(20,5){\color{gris25}{\rule{4mm}{4mm}}}
\put(25,15){\color{gris25}{\rule{4mm}{4mm}}}
%%%%%%%%%%%%%%%%%%%%%%%%%%%%%%%%%%%%%%%% QUADRILLAGE
\put(0,0){\line(1,0){30}}
\put(0,5){\line(1,0){30}}
\put(0,10){\line(1,0){30}}
\put(0,15){\line(1,0){30}}
\put(0,20){\line(1,0){30}}
\put(0,25){\line(1,0){30}}
\put(0,30){\line(1,0){30}}
\put(0,0){\line(0,1){30}}
\put(5,0){\line(0,1){30}}
\put(10,0){\line(0,1){30}}
\put(15,0){\line(0,1){30}}
\put(20,0){\line(0,1){30}}
\put(25,0){\line(0,1){30}}
\put(30,0){\line(0,1){30}}
\end{picture}}
\end{figure}

Finally, we derive from Theorem~\ref{mainthm} and \eqref{defLS} a new combinatorial interpretation for the Legendre-Stirling numbers of the second kind.
The correspondence between the two models will be established in the next subsection.

\begin{corollary} \label{coroPSnk} The integer $\LS(n,k)$ is the number of ordered pairs $(Q_1,Q_2)$ of simply hooked $k$-quasi-permutations of $[n]$ such that $pr_y(Q_1)=pr_y(Q_2)$. \end{corollary}

\begin{remark}
We haven't found an interpretation neither for the numbers $d_{n,k}^{(i)}$  in \eqref{z+1}, nor for the formulas expressed in \eqref{eqankdnk}, in terms of simply hooked quasi-permutations.
\end{remark}

%%%%%%%%%%%%%%%%%%%%%%%%%%%%%
\subsection{The link between the two models}
\label{s:bijection}

We introduce a third interpretation which permits to make the connection easier between the two previous models. Let $\Pi_{n,k}$ be the set of partitions of $[n]$ in $k$ non-empty blocks.

\begin{definition} \label{eq:def4} Let $\mathcal{B}_{n,k}^{(i)}$ be
the set of
 triples $(\pi_1,\pi_2,\pi_3)$ in $\Pi_{n,k+i} \times \Pi_{n,k+i} \times \Pi_{n,n-i}$ such that:
 \begin{itemize}
\item[i)] $\min(\pi_1)=\min(\pi_2)$ and $\Sing(\pi_1)=\Sing(\pi_2)$,
\item[ii)] $\min(\pi_1) \cup \Sing(\pi_3) =\Sing(\pi_1) \cup \min(\pi_3) = [n]$,
\end{itemize}
where $\Sing(\pi)$ denotes the set of singletons in $\pi$.
\end{definition}

We will need the following result.
\begin{lemma}
For $(\pi_1,\pi_2,\pi_3) \in \mathcal{B}_{n,k}^{(i)}$, we have:
\begin{itemize}
\item[i)] $|\min(\pi_1) \cap \min (\pi_3) | = k$,
\item[ii)]$|\Sing(\pi_1) \setminus \min(\pi_3)| = i$,
\item[iii)]$|\Sing(\pi_3) \setminus \min(\pi_1)|=n-k-i$.
\end{itemize}
\end{lemma}
\begin{proof} By definition, we have $|\min(\pi_1)|=k+i$ and $|\min (\pi_3)|=n-i$.
Since $\min(\pi_1)\cup \min (\pi_3)=[n]$, by sieve formula, we deduce
$$
|\min(\pi_1) \cap \min (\pi_3) | =|\min(\pi_1)|+ |\min (\pi_3) |-
 |\min(\pi_1)\cup \min (\pi_3)|=k,
$$
and
 $$
|\Sing(\pi_1)\setminus \min(\pi_3)|=|\Sing(\pi_1)|-|\Sing(\pi_1)\cap  \min(\pi_3)|
=n-|\min(\pi_3)|=i.
$$
In the same way, we obtain iii).
\end{proof}

\begin{theorem} \label{bij}
There is a bijection between $\mathcal{A}_{n,k}^{(i)}$  and $\mathcal{B}_{n,k}^{(i)}$.
\end{theorem}
\begin{proof}
 Let $\pi=\{B_0,B_1,\ldots,B_k\}$ be a signed $k$-partition in $\mathcal{A}_{n,k}^{(i)}$.
 We  construct the triple $(\pi_1,\pi_2,\pi_3)$ of partitions by the following algorithm.
\begin{enumerate}
 \item[$\pi_1, \pi_2$:] \begin{itemize}
        \item  Let $\pi'=\{B_0',B_1',\ldots,B_k'\}$ be the partition obtained by exchanging
        all  $j$  and  $-j$ in $\pi$ if $j \in B_0$ (resp. $j\in [n]$).
\item Let  $\pi''=\{B_0'',B_1'',\ldots,B_k''\}$  be the partition obtained by removing
 all the negative values in $\pi'$.
\item  Define $\pi_1$ (resp. $\pi_2$) to be the partition
obtained by splitting the $i$ positive  elements in $B_0''$  into  $i$  singletons and deleting 0 in $\pi''$.
       \end{itemize}
The resulting  partitions $\pi_1$  and $\pi_2$ are  clearly  elements of $\Pi_{n,k+i}$ and
 satisfy
 $\min(\pi_1)=\min(\pi_2)$ and $\Sing(\pi_1)=\Sing(\pi_2)$.
\item[$\pi_3$:] \begin{itemize}
       \item
For all $p \in [n]\setminus \min \pi$ such that $B_0 \cap \{\pm p \} = \emptyset$, move $p$ into the zero-block and obtain the partition $\pi'=\{B_0',B_1',\ldots,B_k'\}$.
So there are $n-k-i$ positive entries in the new $B_0'$.
\item Let  $\pi''=\{B_0'',B_1'',\ldots,B_k''\}$  be the partition obtained by removing
all the negative values in $\pi'$.
\item Define $\pi_3$  to be the partition
obtained by splitting the $n-k-i$ positive  elements in $B_0''$  into  $n-k-i$  singletons and deleting 0 in $\pi''$.
 \end{itemize}
The resulting partition $\pi_3$  is an element of $\Pi_{n,n-i}$.
\end{enumerate}
For any $p\in [n]\setminus \min(\pi_1)$,  if  $p\notin B_0$  then $B_0\cap \{\pm p\}\neq \emptyset$,
by definition $p$ will be moved in the zero-block, otherwise $p$ is already in the zero-block.
Thus, the elements that are not in $\min(\pi_1)$ become singletons in $\pi_3$. Hence
 $\min(\pi_1) \cup \Sing(\pi_3)=[n]$. Similarly we have $\Sing(\pi_1) \cup \min(\pi_3)=[n]$.

For example, for  the signed 3-partition of $[\pm 10]_0$:
\begin{align}\label{eq:example}
\pi =\{ \{ -4,6,7,-8,-10\}_0,\{\pm1,3,4,-5,-7\},\{\pm2,-3,5,-6,8\},\{\pm9,10\}\},
\end{align}
the corresponding triple is $(\pi_1,\pi_2,\pi_3) \in \Pi_{10,6} \times \Pi_{10,6} \times \Pi_{10,7}$
with : \begin{equation} \label{pi}
\begin{array}{c} \pi_1= \{ \{1,3,7\},\{2,5,6\},\{4\},\{8\},\{9\}, \{10\} \}, \\
 \pi_2= \{ \{1,5,7\},\{2,3,6\},\{4\},\{8\},\{9\}, \{10\} \}, \\
 \pi_3= \{ \{1,4\},\{2,8\},\{3\},\{5\},\{6\},\{7\}, \{9,10\} \}. \end{array}
\end{equation}

Conversely, for any $(\pi_1,\pi_2,\pi_3)\in \mathcal{B}_{n,k}^{(i)}$, we construct $\pi=\{ B_0,B_1,\ldots,B_k \}\in \mathcal{A}_{n,k}^{(i)}$ with the following procedure:
\begin{itemize}
 \item Use the $k$ elements of $\min(\pi_1) \cap \min(\pi_3)$, say $p_1,\ldots, p_k$ and 0
  to create $k+1$  blocks:
  \begin{align}\label{eq:BC}
  B_0=\{\ldots\}_0, \;   B_1=\{\pm p_1, \ldots\},\ldots, B_k=\{\pm p_k\ldots\},
  \end{align}
  where  ``$\ldots$'' means that the blocks are not completed.
   For instance, for the triple $(\pi_1,\pi_2,\pi_3)$ in \eqref{pi}, we create four blocks: $\{0, \ldots \}$, $\{ \pm 1, \ldots\}$, $\{ \pm 2, \ldots \}$ and $\{ \pm 9, \ldots \}$.
\item For each  element $x_j$ of $[n] \setminus \min(\pi_3)$ ($1\leq j \leq i$), suppose that $x_j$  appears in a non-singleton block $C_j$ of $\pi_3$.
Then put $-x_j$  into the zero-block $B_0$ and $x_j$ into the block  in \eqref{eq:BC} that contains $\min(C_j)$.
Note that we must show that   $\min(C_j) \in \min(\pi_1) \cap \min(\pi_3)$ to warrant the existence of such a block in \eqref{eq:BC}.
Indeed, if  $\min(C_j) \notin \min(\pi_1)$, then, by Definition~\ref{eq:def4},  we would have $\min(C_j)\in \Sing(\pi_3)$.
For the current example, we place the number $4$ in the block that contains $1$.
\item For each element $y_j$ of $[n] \setminus \min(\pi_2)$ ($1\leq j \leq n-k-i$), suppose that $y_j$
 appears in a non-singleton block $D_j$  (resp. $E_j$) of $\pi_2$ (resp. $\pi_1$).
 Then put $-p_j$ into the block in \eqref{eq:BC} that contains $\min(D_j)$ and
 put $p_j$ into the block in \eqref{eq:BC} that contains $\min(E_j)$ if this block dosn't  contains $-p_j$, into
the zero-block $B_0$ otherwise.
  For the current example, we place the number $-3$ in the block that contains $2$.
and  $5$ in the block that contains $2$, and $6$ in the zero-block because the block that contains $2$ already has $-6$.
\end{itemize}
\end{proof}

Since $\varphi$ described in Remark~\ref{remark3} maps
each partition to  a subdiagonal quasi-permutation,  for every triple $(\pi_1,\pi_2,\pi_2)$ of partitions satisfying the conditions of Theorem \ref{bij}, we can associate a triple
$(P_1, P_2,P_3)=(\varphi(\pi_1), \varphi(\pi_2), \varphi(\pi_3))$ of subdiagonal quasi-permutations.
  If $\overline{P_i}$ denotes the supdiagonal quasi-permutation obtained from $P_i$ exchanging the $x$ and $y$ coordonates, then $(Q_1,Q_2)=(\overline{P_1} \cup P_3, \overline{P_2} \cup P_3)$ is an ordered pair of simply hooked quasi-permutations satisfying the conditions of Theorem \ref{mainthm}.
Thus, we obtain a bijection between the signed $k$-partitions and the ordered pairs of simply hooked quasi-permutations.

For example,  for the signed $3$-partition $\pi$ in \eqref{eq:example},
the corresponding
 ordered pair of simply hooked quasi-permutations $(Q_1,Q_2)$  is then given by \eqref{eq:qq} (cf.  Figure \ref{couple}).

%%%%%%%%%%%%%%%%%%%%%%%%%%%%%%%%%%%
\section{Jacobi-Stirling numbers of the first kind $js_{n}^{k}(z)$}
% For convenience, we define $\Js_{n}^{k}(z):=(-1)^{n-k}\js_n^k(z)$, then
%\begin{equation}
% \label{eqjs'nk} \Js_{n}^{k}(z) = \Js_{n-1}^{k-1}(z)+(n-1)(n-1+z)\Js_{n-1}^{k}(z), \qquad n,k \geq 1
%\end{equation}

For  a permutation $\sigma$ of  $[n]_0:=[n]\cup \{0\}$ (resp. $[n]$) and
 for $j\in [n]_0$ (resp. $[n]$), denote by
$\Orb_\sigma(j)=\{ \sigma^\ell(j): \ell \geq 1 \}$ the orbit of $j$ and $\min(\sigma)$
  the set of  its cyclic minima, i.e.,
  $$
\min(\sigma)=\{ j \in [n]: j=\min(\Orb_\sigma(j)\cap [n]) \}.
  $$

\begin{definition}  Given a word $w=w(1)\ldots w(\ell)$ on the finite alphabet $[n]$,
a  letter $w(j)$ is a \emph{ record} of $w$ if $w(k)>w(j)$ for every $k\in \{1,\ldots, j-1\}$.
We define $\rec(w)$  to be the number of records of $w$ and $\rec_0(w)=\rec(w)-1$.
\end{definition}

For example, if $w={\bf 5}7{\bf  4}86{\bf 2}3{\bf 1}9$, then
the records are  $5, 4, 2,1$. Hence $\rec(w)=4$.

\begin{theorem} \label{thm3}
The integer $b_{n,k}^{(i)}$ is   the number of ordered pairs $(\sigma,\tau)$ such that
$\sigma$ $($resp. $\tau)$ is a permutation of $[n]_0$ $($resp. $[n])$
with $k$ cycles, satisfying
\begin{itemize}
 \item[i)] $1\in \Orb_\sigma(0)$,
\item[ii)] $\min \sigma=\min \tau$,
\item[iii)] $\rec_0(w)=i$, where $w=\sigma(0)\sigma^2(0)\ldots \sigma^l(0)$ with $\sigma^{l+1}(0)=0$.
\end{itemize}
\end{theorem}

\begin{proof}
Let $\mathcal{E}_{n,k}^{(i)}$ be the set of ordered pairs
$(\sigma,\tau)$  satisfying the conditions of Theorem \ref{thm3} and $e_{n,k}^{(i)}=\left|\mathcal{E}_{n,k}^{(i)}\right|$.
 We divide $\mathcal{E}_{n,k}^{(i)}$ into three parts:
\begin{itemize}
\item[(i)] the ordered pairs $(\sigma,\tau)$ such that $\sigma^{-1}(n)=n$.
Then $n$ forms a cycle in both $\sigma$ and $\tau$ and there are clearly $e_{n-1,k-1}^{(i)}$ possibilities.
\item[(ii)]  the ordered pairs $(\sigma,\tau)$ such that $\sigma^{-1}(n)=0$.
We can construct such ordered pairs by first choosing an ordered pair $(\sigma',\tau')$ in $\mathcal{E}_{n-1,k}^{(i-1)}$ and
then inserting $n$ in $\sigma'$ as image of $0$ (resp. in $\tau'$).
Clearly, there are $(n-1)e_{n-1,k}^{(i-1)}$ possibilities.
\item[(iii)] the ordered pairs $(\sigma,\tau)$ such that $\sigma^{-1}(n)\not\in \{0,n\}$.
We can construct such ordered pairs by first choosing an ordered pair $(\sigma',\tau')$ in $\mathcal{E}_{n-1,k}^{(i)}$ and
then inserting $n$ in $\sigma'$ (resp. in $\tau'$).
Clearly, there are $(n-1)^2e_{n-1,k}^{(i)}$ possibilities.
\end{itemize}
Summing up,  we get the following equation:
  \begin{equation} \label{eqgnk}
  e_{n,k}^{(i)} = e_{n-1,k-1}^{(i)} + (n-1) e_{n-1,k}^{(i-1)} + (n-1)^2 e_{n-1,k}^{(i)}.
  \end{equation}
  By \eqref{eqlnk2},  it is easy to see that  the coefficients $b_{n,k}^{(i)}$ satisfy the same recurrence.
   \end{proof}

We show now how to derive from Theorems~\ref{thm1}  and \ref{thm3} the combinatorial interpretations for the numbers $|\ls(n,k)|$, $|s(n,k)|$ and $|u(n,k)|$.

\begin {corollary}
The integer $|\ls(n,k)|$ is
the number of ordered pairs $(\sigma,\tau)$ such that
$\sigma$ $($resp. $\tau)$ is a permutation of $[n]_0$ $($resp. $[n])$
with $k$ cycles, satisfying
$1\in \Orb_\sigma(0)$
and  $\min \sigma=\min \tau$.
\end {corollary}

\begin{corollary} The integer  $|s(n,k)|$ is the number of  permutations of $[n]$ with
 $k$ cycles.
  \end{corollary}
\begin{proof} By Theorem~\ref{thm3}, the integer  $|s(n,k)|$ is the number of ordered pairs $(\sigma,\tau)$
in  $\mathcal{E}_{n,k}^{(n-k)}$. Since $\sigma$ and $\tau$ both have $k$ cycles with same cyclic minima, the permutation $\sigma$ is completely determinated by $\tau$ because $\Orb_\sigma(1)$ is the only non singleton cycle,
of cardinality $n-k+2$, so the $n-k$ elements different from 0 and 1 are exactly the elements of $[n]\setminus \min \tau$ arranged in decreasing order in the word
$w=\sigma(0)\sigma^2(0)\ldots 1$ with $\sigma(1)=0$. \end{proof}

The following result is the analogue interpretation to Corollary \ref{coroUnk} for the central factorial numbers of the first kind. This analogy is comparable with that of Stirling numbers of the first kind $|s(n,k)|$ versus the Stirling numbers of the second kind $|S(n,k)|$.

\begin{corollary} \label{thm53} The integer $|u(n,k)|$ is the number
 of ordered pairs $(\sigma,\tau) \in \mathcal{S}_n^2$  with $k$ cycles, such that  $\min(\sigma) = \min(\tau)$.
%\item[ii)] $\sigma_1(n)=n \Leftrightarrow \sigma_2(n)=n$,
\end{corollary}
Indeed,  the integer $|u(n,k)|$ is the number of ordered pairs $(\sigma,\tau)$ in
 $\mathcal{E}_{n,k}^{(0)}$. Theorem~\ref{thm3} implies that $\sigma^{-1}(1)=0$.
The result follows then by deleting the zero in $\sigma$.

\begin{remark} By the substitution $i \rightarrow n+1-i$, we can derive that the number $|u(n,k)|$
is also the number of ordered pairs $(\sigma,\tau)$ in $\mathcal{S}_n^2$ with $k$ cycles,
such that  $\max(\sigma) = \max(\tau)$,
 where $\max(\sigma)$ is the set of cyclic maxima of $\sigma$, i.e.,
 $$ \max(\sigma) = \{ j \in [n]:  j=\max(\Orb_\sigma(j)  \}.$$ \end{remark}

%\todo{More details}

\section{Further results}
\subsection{Central factorial numbers of odd indices}

For all $n,k \geq 0$, set
$$
V(n,k)=4^{n-k} T(2n+1,2k+1),\quad v(n,k)=4^{n-k} t(2n+1, 2k+1).
$$
 Note that these numbers are also integers (see Table \ref{U2n12k1}).  By definition, we have the following recurrence relations~:
\begin{align}
V(n,k) &= V(n-1,k-1) + (2k+1)^2 V(n-1,k),\label{eqVnk}\\
 v(n,k) &= v(n-1,k-1) - (2n-1)^2 v(n-1,k). \label{eqvnk}
 \end{align}

%%%%%%%%%%%%%%%%%%%%%%%
\begin{table}
\caption{The first values of $V(n,k)$ and $|v(n,k)|$}
\label{U2n12k1}
{\scriptsize
 \[ \begin{tabular}{c|cccccc}
$k\backslash n$ & $0$ & $1$ & $2$ & $3$ & $4$ & $5$
\\
\hline
$0$ & $1$ & $1$& $1$  & $1$  &$1$    &$1$ \\
$1$ &     & $1$& $10 $&$91$  &$820$  &$7381$   \\
$2$ &     &    & $1$ &$35$&$966$ &$24970$   \\
$3$ & &&             &$1$ &  $84$ &$5082$                \\
$4$ & &&&&$1$  &$165$        \\
$5$ & & &&&&$1$ \\
 \end{tabular}
\qquad
  \begin{tabular}{c|cccccc}
$k\backslash n$ & $0$ & $1$ & $2$ & $3$ & $4$ & $5$
\\
\hline
$0$ & $1$ & $1$& $9$  & $225$  &$11025$    &$893025$ \\
$1$ &     & $1$& $10 $&$259$  &$12916$  &$1057221$   \\
$2$ &     &    & $1$ &$35$&$1974$ &$172810$   \\
$3$ & &&             &$1$ &  $84$ &$8778$                \\
$4$ & &&&&$1$  &$165$        \\
$5$ & & &&&&$1$ \\
 \end{tabular} \] }
  \end{table}
The natural question is to find a combinatorial interpretation for these numbers.
We can easily find it from combinatorial theory of generating functions.

\begin{theorem} \label{thm52} The integer $V(n,k)$ is the number of partitions of $[2n+1]$ into  $2k+1$ blocks of odd cardinality.
\end{theorem}
\begin{proof} This follows from the  known generating function (see \cite[p. 214]{Riordan}):
$$\sum\limits_{n,k \geq 0} V(n,k) t^k \frac{x^n}{n!} = {\sinh( t \sinh(x))},
$$ and the classical combinatorial theory of generating functions (see \cite[Chp. 3]{FS72} and \cite[Chp. 5]{Stanley}). \end{proof}

To interpret the integer $|v(n,k)|$,  we need to introduce  the following definition.
 \begin{definition}
A \emph{$(n,k)$-Riordan
 complex} is a $(2k+1)$-tuple
 $$
 ((B_1, \sigma_1, \tau_1), \ldots, (B_{2k+1}, \sigma_{2k+1}, \tau_{2k+1}))
 $$ such that
 \begin{itemize}
 \item[i)] $\{B_1, \ldots, B_{2k+1}\}$ is a partition of $[2n+1]$ into blocks $B_i$ of odd cardinality;
  \item[ii)] $\sigma_i$
  and $\tau_i$ ($1\leq i\leq 2k+1$) are  fixed point free involutions
  on $B_i \backslash \max(B_i)$.
  \end{itemize}
 \end{definition}

\begin{theorem} \label{54}
 The integer $|v(n,k)|$ is the number of  $(n,k)$-Riordan
 complexes.
\end{theorem}
\begin{proof}
It is known that (see \cite[p. 214]{Riordan}):
$$\sum\limits_{n,k \geq 0} |v(n,k)| t^k \frac{x^n}{n!} = \sinh(t \arcsin(x)),$$
and \[ \arcsin(x) =  \sum\limits_{n \geq 0} \left( (2n-1)!! \right)^2 \frac{x^{2n+1}}{(2n+1)!}, \] where $(2n-1)!!=(2n-1)(2n-3)\cdots 3\cdot 1$. Since $(2n-1)!!$ is the number of involutions without fixed points on $[2n]$ (see \cite{Comtet74}),  the integer $((2n-1)!!)^2$ is the number of ordered pairs of involutions without fixed points on $[2n+1] \backslash \{2n+1\}$.

Define the numbers $J(n,m)$  by:
\[ \exp \left( t \sum\limits_{n \geq 1} \left( (2n-1)!! \right)^2 \frac{x^{2n+1}}{(2n+1)!} \right)
= \sum\limits_{n,m \geq 0} J(2n+1,m) t^m \frac{x^{2n+1}}{(2n+1)!}.
 \]
Then, by the theory of exponential generating functions (see \cite[Chp. 3]{FS72} and \cite[Chp. 5]{Stanley}), the coefficient $J(2n+1,m)$ is the number of $m$-tuples $$(B_1, \sigma_1,\tau_1), \ldots, (B_m,\sigma_m,\tau_m);$$ where $\{B_1,\ldots,B_m \}$ is a partition of $[2n+1]$ with $|B_i|$ odd ($1\leq i \leq m$), and $\sigma_i$ and $\tau_i$ are involutions without fixed points on $B_i \backslash \max(B_i)$. As $\sinh(x)=(e^x - e^{-x})/2$, we have $|v(n,m)|=J(2n+1,2k+1)$ if $m=2k+1$, and $|v(n,m)|=0$ if $m$ is even.
\end{proof}

\begin{remark} From \eqref{eqtnk}, we can easily deduce that
\begin{equation}\label{eqstanley}
\sum\limits_{k=0}^{n} |v(n,k)| t^{2k+1} = t(t^2+1)(t^2+3^2)\ldots(t^2+(2n-1)^2)
\end{equation}
 It is interesting to note that a proof of the latter result is not obvious from \eqref{eqstanley}. In the same way, proofs for Theorems~\ref{thm52} and \ref{54} by using \eqref{eqVnk} or  \eqref{eqvnk}  are not obvious.
\end{remark}

\begin{example}There are  ten $(2,1)$-Riordan complexes. Since the numbers $n$ and $k$ are small, the involved involutions are identical transpositions. \[ \begin{array}{cc}
\{ 1 \},\{ (2,3),4\},\{5\}, & \{1 \}, \{2\},\{ (3,4), 5 \}, \\
\{ (1,2),3 \},\{4\},\{5\}, & \{(1,2),5 \},\{3\},\{4\}, \\
\{ (1,3),4\},\{2\},\{5\}, & \{(1,3),5 \},\{2\},\{4\}, \\
\{ (1,2),4\},\{3\},\{5\}, & \{ (1,4),5\},\{2\},\{3\}, \\
\{ 1\},\{(2,3),5\},\{4\}, & \{ 1\},\{(2,4),5\},\{3\}, \end{array}\]
where  $\{ 1 \},\{ (2,3),4\},\{5\}$ means that $\pi=\{ \{ 1 \},\{ 2,3,4\},\{5\} \}$, and $\sigma=\tau=13245$.\end{example}
%The other lines of Table p.226 in \cite{Comtet74} are also very close to line $6$, and there is maybe a link with central factorial numbers, Stirling numbers of more generally with coefficients of Jacobi-Stirling numbers.

\subsection{Generating functions}

In \cite{Everitt2}, the authors made a long calculation to derive an explicit formula for the Jacobi-Stirling numbers.
  Actually, we can derive an explicit formula for the Jacobi-Stirling numbers
  straightforwardly from the Newton interpolation formula:
 \begin{align}\label{eq:newton}
 x^n=\sum_{j=0}^n\left(\sum_{r=0}^j\frac{x_r^n}{\displaystyle \prod_{k\neq i}(x_r-x_k)}\right)\prod_{i=0}^{j-1}(x-x_i).
 \end{align}
 Indeed, making the substitutions
 $x\to m(z+m)$ and $x_i\to i(z+i)$
 in  \eqref{eq:newton}, we obtain
 \begin{align}\label{eq:base}
 \left(m(m+z)\right)^n=\sum_{j=0}^n\JS_n^j(z)(m-j+1)_j(z+m)_j,
 \end{align}
 where
  \begin{align} \label{explicit}
 \JS_n^j(z)=\sum_{r=0}^j(-1)^r\frac{[r(r+z)]^n}{r!(j-r)!(z+r)_r(z+2r+1)_{j-r}},
 \end{align}
 and $(z)_n = z(z+1)\ldots(z+n-1)$.
\begin{remark} If we substitute $x$ by $m(m+z)+k$, we obtain \cite[Theorem 4.1]{Everitt2}. \end{remark}

From the recurrence  \eqref{eqlnk1}, we derive:
 \begin{equation}
 \sum\limits_{n \geq k} \JS_n^k(z) x^n=\frac{x}{1-k(k+z)}\sum\limits_{n \geq k-1} \JS_n^{k-1}(z) x^n;
\end{equation}
therefore,
 \begin{equation}
\sum\limits_{n \geq k} \JS_n^k(z) x^n= \frac{x^k}{(1-(z+1)x)(1-2(z+2)x)\ldots(1-k(z+k)x)}.
\end{equation}

\section*{Acknowledgement} Ce travail a b\'en\'efici\'e d'une aide de l'Agence Nationale de la Recherche portant la r\'ef\'erence ANR-08-BLAN-0243-03.

\bibliographystyle{amsplain}

\end{document}